\documentclass[12pt]{article} 

\usepackage{times}
\usepackage{latexsym}
\usepackage{amsmath}
\usepackage{amsfonts}
\usepackage{xcolor}
\usepackage{tikz}
\usepackage{amsthm}
\begingroup
    \makeatletter
    \@for\theoremstyle:=definition,remark,plain\do{%
        \expandafter\g@addto@macro\csname th@\theoremstyle\endcsname{%
            \addtolength\thm@preskip\parskip
            }%
        }
\endgroup
\makeatletter
\newtheorem*{rep@theorem}{\rep@title}
\newcommand{\newreptheorem}[2]{%
\newenvironment{rep#1}[1]{%
 \def\rep@title{#2 \ref{##1}}%
 \begin{rep@theorem}}%
 {\end{rep@theorem}}}
\makeatother
\usepackage{parskip}
\usepackage{epstopdf}
\usepackage{bm}
\usepackage{breqn}

\newtheorem{Definition}[equation]{Definition}
\newcounter{mark}
\newtheorem{Mark}[mark]{Marking Scheme}
\newtheorem{Walk}[mark]{Random Walk}

\newtheorem{Theorem}[equation]{Theorem}
\newtheorem{Proposition}[equation]{Proposition}
\newtheorem{Lemma}[equation]{Lemma}
\newtheorem{Corollary}[equation]{Corollary}
\newtheorem{Remark}[equation]{Remark}

\theoremstyle{definition}
\newtheorem{Example}[equation]{Example}
\newreptheorem{Theorem}{Theorem}
\newreptheorem{Mark}{Marking Scheme}

\DeclareMathOperator{\Pre}{\Pr_{ex}}

\newcommand{\tp}[2]{(#1 \; #2)}

\newcommand{\al}{\alpha}

\newcommand{\la}{\lambda}

\newcommand{\eps}{\epsilon}
\definecolor{darkred}{rgb}{0.7,0,0} 
\newcommand{\darkred}{\color{darkred}} 
\newcommand{\defn}[1]{\emph{\darkred #1}} 
\newcommand{\f}[2]{\frac{#1}{#2}}

\title{A strong stationary time for random transpositions}
\author{Graham White}
\date{\today}
\begin{document}

\maketitle

\abstract{We show that the random transposition walk on the symmetric group $S_n$ has cutoff in separation distance at $\frac{1}{2}n \log n$, by constructing a strong stationary time. The construction involves working with cycle types of permutations and some partition combinatorics.}

\section{Introduction}
\label{sec:intro}

The random transposition walk on the symmetric group $S_n$ has been extensively studied, for instance as the setting for the pioneering work of Diaconis and Shashahani in \cite{DS}. In \cite{Matthews}, Matthews shows that the separation distance mixing time is asymptotically at most $\frac{1}{2}n \log n$, which matches the standard coupon collector lower bound and thus gives cutoff for this walk in separation distance. It is shown in \cite{GWThesis} that the proof of this upper bound has a subtle flaw, and so the proof of separation distance cutoff is incomplete. The present paper presents an alternative proof of the upper bound, recovering separation distance cutoff for the random transposition walk.

Our main result is the following.

\begin{Theorem}
The random transposition walk has cutoff in separation distance at $\frac{1}{2}n \log n$.
\end{Theorem}

We prove this theorem in Section \ref{sec:sst}, by constructing a strong stationary time for the (lazy) random transposition walk. Calculations supporting part of the proof are given in Section \ref{sec:calcs}. 

\subsection*{Acknowledgements}

I would like to thank Persi Diaconis and Megan Bernstein for many helpful discussions.

\subsection{History of the problem}

The random transposition walk is the following:

\begin{Walk}\label{wal:rt}
The random transposition walk on the symmetric group $S_n$ is the random walk where each step is either any transposition $\tp{i}{j}$, each with probability $\frac{2}{n^2}$, or the identity, with probability $\frac{1}{n}$. Customarily, this is viewed as a shuffling procedure on a deck of $n$ cards, where each step consists of choosing a random card with each hand and swapping those cards, doing nothing when the same card is chosen twice. We start at the identity permutation.
\end{Walk}

In $\cite{DS}$, Diaconis and Shashahani prove that the random transposition walk has cutoff in total variation distance at time $\frac{1}{2}n\log(n) \pm cn$. To show that this walk also has cutoff in separation distance, it is necessary to give an appropriate upper bound on the separation distance mixing time. In \cite{Broder}, Broder uses a strong stationary time to show that the separation distance mixing time is at most $2n \log n$. In \cite{Matthews}, Matthews improved this to $n \log n$, and then gave another modification which attempted to further improve the bound to $\frac{1}{2}n\log n$. This latter modification contains a subtle error, which was identified in \cite{GWThesis}. There, the author considers statements implied by Matthews' techniques which are stronger than those used for his result, and shows that these are false. An upcoming paper will discuss this issue further and give an explicit calculation showing that the error is not just with the proof but with some of the results, and thus that it cannot be easily fixed. Our goal in this paper is to recover the cutoff result via an alternate proof of this upper bound.

We will briefly present here Matthews' first improvement to Broder's scheme. Terminology has been chosen for ease of exposition, rather than being the same as in the original papers.  

\begin{Mark}[Due to Broder, \cite{Broder}, improvement by Matthews, \cite{Matthews}]
\label{mar:sstoriginal}
As Random Walk \ref{wal:rt} runs, mark cards as follows. At the first step, mark the chosen cards. At each later step, if an unmarked card and a marked card are chosen, or if an unmarked card is chosen twice, then mark the unmarked card. 
\end{Mark}

Broder and Matthews show that this marking scheme produces a strong stationary time, in the sense that conditioning on the event that every card has been marked after $t$ steps, for any $t$, produces a uniform distribution on elements of $S_n$. The key observation is that at each step, the set of marked cards are equally likely to be in any order amongst themselves. When a new card is marked, it is equally likely to have been swapped with any of the $k$ already-marked cards or to have been left in place, resulting in $k+1$ cards which are equally likely to be in any order.

Matthews attempted to improve this further by combining it with another strong stationary time, but a subtle mismatch between the assumptions in the inductions for the two strong stationary times meant that the result was not actually a strong stationary time. See Section $5.3$ of \cite{GWThesis} for discussion of this error.

Marking Scheme \ref{mar:sstoriginal} marks cards slowly when there are only a few marked cards, or when there are only a few unmarked cards. Marking cards slowly when there are only a few unmarked cards seems reasonable, because this corresponds to needing to choose each card in the deck at least once, and it will take some time to choose the last few cards. However, it seems plausible that the first few cards moved should be somehow close to random, and this idea is not captured by this marking scheme, which takes $O(n)$ steps to mark each of the first few cards.

Thus, improvements to Broder's strong stationary time will likely need to focus on the early stages of the process. It takes $\frac{1}{2}n\log(n)$ steps to mark the first half of the cards and another $\frac{1}{2}n\log(n)$ steps to mark the other half. If we could design a similar scheme which marks half of the cards in $O(n)$ steps, and the remaining half in $\frac{1}{2}n\log(n)$ steps, this would give the desired upper bound. This is what we will do in the next section.

\subsection{Notation}

We will consider transpositions to act on card labels, rather than on card positions. For instance, the transposition $\tp{1}{2}$ swaps the cards labelled by $1$ and $2$, not the cards currently in the first and second positions. (Because the set of transpositions is a conjugacy class, either choice would be valid).

We will often encounter partitions as the cycle types of permutations. Partitions will be written in descending order, for instance $(4,2,1)$ or $(7)$ are two partitions of $7$. Unless otherwise specified, these are integer partitions rather than set partitions, though we will sometimes need to associate additional data to these partitions, data which takes them partway towards being set partitions. 

We will be interested in the cycle types of uniformly random permutations.

\begin{Lemma}\label{lem:cycletype}
Let $\la$ be a partition of $n$, with for each $i$, $a_i$ the number of parts of size $i$ in $\la$. Then the probability that a uniformly random element of $S_n$ has cycle type $\la$ is $$\Pr(\la) = \frac{1}{\prod_{i=1}^{n}i^{a_i}(a_i)!}.$$
\end{Lemma}

We will find it more convenient to work with a lazier version of the random transposition walk. 

\begin{Walk}\label{wal:rt2}
A lazier version of Random Walk \ref{wal:rt} is to apply any transposition, each with probability $\frac{1}{n(n-1)}$, or to do nothing, with probability $\frac{1}{2}$. This may be seen as choosing a random transposition, flipping a coin, and either applying that transposition or doing nothing.
\end{Walk}

For the remainder of the paper, we will work with Random Walk \ref{wal:rt2}. The question of whether Walk \ref{wal:rt} has cutoff in separation distance at $\frac{1}{2}n \log(n)$ is equivalent to whether Walk \ref{wal:rt2} has cutoff (in separation distance) at $n \log(n)$, because $n \log(n)$ steps of Walk \ref{wal:rt2} or $\frac{1}{2}n \log(n)$ steps of Walk \ref{wal:rt} both include $\frac{1}{2}n \log(n) + o(n)$ non-identity steps. 

As we run the random walk, we use the following notation for the steps taken and the current permutation.

\begin{Definition}\label{def:walknot}
As random walk \ref{wal:rt2} runs, let $\tau_t$ be the transposition chosen for the $t$th step, $\al_t$ be either $1$ or $0$ according to whether the transposition $\tau_t$ was actually applied or not, and $\pi_t$ be the permutation of the cards after this $t$th step.
\end{Definition}

Working with the random transposition walk, we will need to consider the number of transpositions required to build a given permutation.

\begin{Definition}
If $\pi$ is a permutation, then $l(\pi)$ is the minimum number of transpositions which can be multiplied to produce $\pi$. This is the length of the shortest path between the identity vertex and the vertex $\pi$ in the Cayley graph of $S_n$ with generating set the set of all transpositions.  
\end{Definition}

The length $l(\pi)$ may be computed by observing that when $\pi$ is a $k$--cycle, $l(\pi) = k-1$, and that $l$ is additive on disjoint cycles. This also implies that $l(\pi)$ depends only on the cycle type of $\pi$, so is unchanged by conjugation.

For any permutation $\pi$ and any transposition $\tau$, the length $l(\pi\tau)$ is either $l(\pi) - 1$ or $l(\pi) + 1$, depending on whether the two elements interchanged by $\tau$ are in the same cycle of $\pi$ or not. The same is true of $l(\tau\pi)$. 

\subsection{Merging partitions}

Our construction of a strong stationary time will rely on combining partitions in a certain way. The details of this section may safely be skipped on a first reading.

\begin{Definition}\label{def:merge}
Let $\la$ and $\mu$ be partitions of $m$ and $n$, respectively, with $m \geq n$. To \defn{merge} $\mu$ with $\la$ is to do the following.

Create new partitions $\nu$ and $\xi$ by distributing the parts of $\mu$ between $\la$ and an empty partition, as follows.
\begin{itemize}
\item Initialise $\nu$ as equal to $\la$ and $\xi$ as equal to $\mu$.
\item Choose $\mu_0$ to be a random part of $\mu$, with probability proportional to part size. 
\begin{itemize}
\item Case 1: With probability $\frac{m}{m+1}$, increase a random part $\nu_0$ of $\nu$ by $\mu_0$, again choosing proportionally to part size. 
\item Case 2: Otherwise (probability $\frac{1}{m+1}$), adjoin $\mu_0$ to $\nu$ as a new part. 
\end{itemize}
\item In either case, remove $\mu_0$ from $\xi$.
\item With probability $\frac{|\xi|}{|\nu|}$, choose a random part of $\xi$ with probability proportional to part size, remove it from $\xi$, append it to $\nu$, and repeat this step (updating the sizes $|\xi|$ and $|\nu|$ used to calculate the probability). Otherwise finish.
\end{itemize}
The resulting partitions $\nu$ and $\xi$ are partitions of random integers --- $\nu$ of an integer between $m$ and $m+n$, and $\xi$ of an integer between $0$ and $n$. The sizes of $\nu$ and $\xi$ add to $m+n$.
\end{Definition}

We will reserve the use of the word `merge' for this operation, using the word `combine' in the more general sense where two parts of sizes $a$ and $b$ become a single part of size $a+b$. The variables $\la$, $\mu$, $\nu$, and $\xi$ will always take these same roles. 

It will be convenient to have a name for the probabilities arising in Definition \ref{def:merge}.

\begin{Definition}\label{def:include}
For any partition $\mu$, any choice $\mu_0$ of a part of $\mu$, and any integer $m \geq |\mu|$, define the function $f(m,\mu,\mu_0)$ to be the sum over all permutations of the parts of $\mu$, whose sizes are $\mu_0$ through $\mu_i$, of the product $$\prod_{j=1}^i\frac{\mu_j}{m+\mu_0+\sum_{l=0}^{j-1}\mu_l}.$$ This is the probability that when the partition $\mu$ is merged with a partition $\la$ of size $m$ according to Definition \ref{def:merge} and starting with a part of size $\mu_0$ that all remaining parts of $\mu$ are added to $\nu$ rather than remaining in $\xi$. 

We call this the including factor of $(\mu,\mu_0)$ with respect to $m$. We will abuse notation slightly and use ``the including factor of $(\mu,\mu_0)$ with respect to $\la$'' to mean ``the including factor of $(\mu,\mu_0)$ with respect to $|\la|$''.
\end{Definition}

\begin{Remark}\label{rem:include}
Definition \ref{def:include} may also be used to describe the probability that a certain subset of the parts of $\mu$ are added to $\nu$ and the others are not. If $\mu'$ is any subpartition of $\mu$, working for the moment with labelled partitions where parts of the same size may be distinguished from one another, then the probability that the parts in $\mu'$ are added to $\nu$ and the other parts of $\mu$ remain in $\xi$ is $f(m,\mu',\mu_0)\cdot\frac{|\mu|-|\mu'|}{m+|\mu'|}$. Returning to unlabelled partitions, let $b_p$ and $c_p$ be the number of parts of size $p$ in $\mu'$ and in the remainder of $\mu$, respectively, not counting the part $\mu_0$. Then when $\mu$ of size $n$ is merged with a partition $\la$ of $m$, the probability that the parts added to $\nu$ are described exactly by $\mu'$ is $$f(m,\mu,\mu')\frac{m-n+2|\mu'|}{m+|\mu'|}\prod_{i=1}^n \binom{b_i+c_i}{b_i}.$$

In later calculations, $I(b)$ will often refer to an including factor of this kind, and $k$ will be used for $|\mu'|$.
\end{Remark}

\begin{Remark}
The probabilities in Definition \ref{def:merge} were obtained from examining the probabilities that certain cycles appear in a random permutation. For instance, a random element of $S_n$ has a probability of $\frac{1}{n}$ of having $1$ as a fixed point. Given that it does not fix $1$, it has a probability of $\frac{1}{n-1}$ that $1$ occurs as part of a $2$--cycle. Given that neither of these is the case, there is a probability of $\frac{1}{n-2}$ that $1$ occurs as part of a $3$--cycle, and so on.

We will construct a strong stationary time where elements of $S_m$ and $S_n$ are combined to produce elements of $S_{m+n}$. The probabilities in Definition \ref{def:merge} are chosen so that the distribution of resulting elements of $S_{m+n}$ has these properties. This is necessary because our goal is to show that under certain circumstances these elements are uniformly distributed in $S_{m+n}$.
\end{Remark}

\section{A strong stationary time}
\label{sec:sst}

In this section, we construct a strong stationary time for the random transposition walk. This strong stationary time may be thought of as a version of Broder's strong stationary time which keeps track of more cards, and as a result is able to mark cards more quickly near the beginning of the process. The proof that this scheme does result in a strong stationary time relies on some detailed calculations involving combining partitions, which we defer until Section \ref{sec:calcs}.

\begin{Definition}\label{def:parts}
In the following, we will want to consider a set partition of $[n]$ into subsets, which evolves with time. For each time $t$, let $P(t)$ be a partition of $[n]$. For each $i$, let $P_i(t)$ refer to the part of $P(t)$ containing $i$.

We will want to compare the sizes of various parts. If two parts have the same size, then we will break ties according to the smallest entry in that part, so $\{2,4,7\}$ is smaller than $\{3,5,6\}$.
\end{Definition}

We will define the partition $P(t)$ in terms of the path taken by an instance of the (lazy) random transposition walk, as described by the variables $\tau_t$, $\pi_t$, and $\al_t$ of Definition \ref{def:walknot}. 

\begin{Mark}\label{mar:mult}
Initially, let the partition $P(0)$ be comprised of $n$ parts of size $1$. 

We now define $P(t)$ in terms of $P(t-1)$ and $\tau_t$, $\pi_t$, and $\al_t$. Let $\tau_t$ be the transposition $\tp{i}{j}$.

If $i$ and $j$ are in the same part of $P(t-1)$, set $P(t) = P(t-1)$. Otherwise, let $P_i(t-1)$ be smaller than $P_j(t-1)$, breaking ties as in Definition \ref{def:parts}. If $\al_t = 1$ or if $\al_t = 0$ and $j$ is the smallest number in $P_j(t-1)$, then define the partition $P(t)$ as follows. Otherwise, set $P(t) = P(t-1)$.

It will be proven in Proposition \ref{prop:unioncycles} that $P_i(t-1)$ and $P_j(t-1)$ are unions of cycles of $\pi_{t-1}$, a fact we will now use.

\begin{itemize}
\item Initialise a partition $P'$ as equal to $P(t-1)$.
\item Consider the permutation $\pi_{t-1}$, and move all elements of the cycle containing $i$ from $P'_i$ to $P'_j$.
\item With probability $\frac{|P'_i|}{|P'_j|}$, choose a random cycle of $\pi_{t-1}$ from $P'_i$ with probability proportional to cycle size, move it from $P'_i$ to $P'_j$, and repeat this step. Otherwise go to the next step. After each iteration, update the sizes of the parts $P'_i$ and $P'_j$ for the calculation of the next probability, but do not change which parts these terms refer to --- that is, $P'_i$ is the part which contained $i$ at the start of this step, even though the $i$ has been moved to a different part.
\item Set $P(t)$ as equal to $P'$.
\end{itemize} 
\end{Mark}

Notice the similarity of this scheme to the definition of merging one partition into another (Definition \ref{def:merge}). That definition was created so that we may analyse this scheme in Section \ref{sec:calcs}.

For the definition of the partition $P(t)$ to make sense, we need the following fact.

\begin{Proposition}\label{prop:unioncycles}
As we run Marking Scheme \ref{mar:mult}, each part $P_i(t)$ is a union of cycles of the permutation $\pi_t$.
\end{Proposition}
\begin{proof}
This is true for $t = 0$ because $P(t)$ has $n$ parts of size $1$. 

Assume that for each $i$, $P_i(t-1)$ is a union of cycles of $\pi_{t-1}$. By definition, $\pi_t = \pi_{t-1}\tau_t^{\al_t}$. 

If $\al_t = 0$, then $\pi_t$ and $\pi_{t-1}$ have the same cycles, and $P(t)$ was obtained from $P(t-1)$ by possibly combining two parts, and possibly moving some of these cycles from one part to another. This results in $P(t)$ being a union of cycles of $\pi_t$.

If $\al_t = 1$, then the cycles of $\pi_t$ are obtained from those of $\pi_{t-1}$ by either splitting one cycle in two or combining two cycles. In the former case, $P(t) = P(t-1)$, so if parts of this partition are unions of cycles of $\pi_{t-1}$, then they are unions of the finer partition whose parts are the cycles of $\pi_t$. In the latter case, only two cycles of $\pi_{t-1}$ are combined in $\pi_t$, and Marking Scheme \ref{mar:mult} guarantees that those two cycles are in the same part of $P(t)$.
\end{proof}

\begin{Proposition}\label{prop:mixed}
As we run Marking Scheme \ref{mar:mult}, the cards in each part $P_i$ are random. More precisely, among paths of length $t$ which result in any given partition $P(t)$, any permutations obtained by arbitrarily rearranging the cards in some or all of the parts $P_i(t)$ are equally likely.  
\end{Proposition}
\begin{proof}
This is true at time $t = 0$, because each part $P_i(0)$ has size $1$ and there are no such rearrangements available.

Assume the result for time $t-1$. Let $\tau_t$ be the transposition $\tp{i}{j}$. There are several cases:
\begin{itemize}
\item The cards $i$ and $j$ are in the same part of $P(t-1)$. In this case, the distribution of the permutation $\pi_t$ is obtained from the distribution $\pi_{t-1}$ by multiplying by $\tp{i}{j}$. But for any permutation $\pi'$, $\pi_{t-1}$ is equally likely to be $\pi'$ as $\pi'\tp{i}{j}$, so the distribution of $\pi_t$ is the same as the distribution of $\pi_{t-1}$. In this case, $P(t)$ is equal to $P(t-1)$, and the claim is true.  
\item The cards $i$ and $j$ are in different parts of $P(t)$, $\al_t = 0$, and $P(t) = \mbox{$P(t-1)$}$. As with the previous case, the distribution of $\pi_t$ is the same as the distribution of $\pi_{t-1}$, and $P(t)$ is equal to $P(t-1)$.
\item The cards $i$ and $j$ are in different parts of $P(t)$, and $P(t)$ is not equal to $P(t-1)$. In this case, permutations obtained by rearranging cards only in parts other than $P_i(t-1)$ and $P_j(t-1)$ will still be equally likely, because such rearrangements commute with the transposition $\tp{i}{j}$. It remains to check that the distribution of $\pi_t$ is unchanged by rearranging cards in $P_i(t)$ (which is equal to $P_j(t)$, given that the distribution of $\pi_{t-1}$ was unchanged by rearranging cards in $P_i(t-1)$ and $P_j(t-1)$.

Checking this fact involves detailed calculations with the cycle types of uniformly random permutations, and is done in Section \ref{sec:calcs}. Proposition \ref{prop:combineparts} shows that the order of the elements of $P_i(t)$ according to $\pi_t$ has the correct distribution of cycle types to be a uniform distribution on all permutations of $P_i(t)$. The random transposition walk is generated by a conjugacy class, and for any permutation $\pi'$ of elements of $P_i(t)$, conjugating the entire random walk path by $\pi'$ preserves the partition $P(t)$ and conjugates $\pi_t$ by $\pi'$, so the distribution of $\pi_t$ conditioned on $P(t)$ is invariant under rearranging elements of $P_i(t)$, as required.
\end{itemize}
Subject to Proposition \ref{prop:combineparts}, this completes the proof.
\end{proof}

Notice that the change from considering partitions to permutations in the proof of Proposition \ref{prop:mixed} means that this proposition need not be true if we condition on the order in which cards are marked. For instance, while the paths which produce $P_1(3) = \{1,2,3,4\}$ in three steps are equally likely to produce any permutation of those four cards, the paths which produce $P(1) = (1,2)(3)(4)\cdots$ after one step, $P(2) = (1,2)(3,4)\cdots$ after a second step, and $P_1(3) = \{1,2,3,4\}$ after a third step do not have this property. Indeed, it is impossible for such a path to produce the permutation $(1 \; 3 \; 2 \; 4)$. Permutations like this one come from paths which build the set $P_1(3) = \{1,2,3,4\}$ in a different order.

\begin{Corollary}\label{cor:sst}
The time taken for Marking Scheme \ref{mar:mult} to produce a partition $P$ with only one part is a strong stationary time. That is, conditioned on $P(t) = [n]$, the distribution of $\pi_t$ is uniform on $S_n$.
\end{Corollary}
\begin{proof}
This is an immediate consequence of Proposition \ref{prop:mixed}.
\end{proof}

Corollary \ref{cor:sst} shows that if we can bound the time taken until the partition $P(t)$ has only one part, then we will have a bound on the mixing time. Our next task is to analyse the time taken for this to happen. This cannot be faster than $n \log(n) + O(n)$, because that's how long it takes for each card to be moved, by a coupon collector calculation. Showing that this amount of time is enough proves our main result.

We first show that time $O(n)$ is enough to find a set of $\frac{n}{3}$ cards which are equally likely to be in any order (in those same positions, independent of the positions of other cards). This is the desired speed-up of Broder's approach, which takes $\frac{1}{2}n \log(n) + O(n)$ steps to get to this point. In the language of Marking Scheme \ref{mar:mult}, this happens when any part of $P(t)$ has size at least $\frac{n}{3}$.

\begin{Proposition}\label{prop:firsthalf}
The time taken for any one of the parts $P_i$ to grow to size at least $\frac{n}{3}$ is $O(n)$, in the sense that for any $\eps$, there is a constant $c$ so that after time $cn$, the probability that no part $P_i$ has ever been as large as $\frac{n}{3}$ is at most $\eps$.
\end{Proposition}
\begin{proof}
Consider the length $l_t = l(\pi_t)$ of the permutation of the cards at each step. While all parts have size less than $\frac{n}{3}$, the probability that the next step splits a cycle (reducing $l$ by one) is at most $\frac{1}{6}$, and the probability that the next step combines two cycles (increasing $l$ by one) is at least $\frac{1}{3}$. Otherwise, the walk doesn't move and $l$ is unchanged.

This means that while no part has grown to size $\frac{n}{3}$, the length $l$ is at least as large as the random walk $X_t$ which starts at $0$ and adds $1, 0,$ or $-1$ with probabilities $\f{1}{3}, \f{1}{2},$ and $\f{1}{6}$. This walk drifts upward at speed $\frac{1}{6}$ --- after $t$ steps, $X_t$ has expectation $E\f{t}{6}$ and variance $\f{17n}{36}$. Therefore, for any $\eps$, there is a constant $c$ so that after $cn$ steps, $X_{cn}$ has a probability of at least $(1-\eps)$ of being larger than $n$. But the quantity $l_t$ is at most $n-1$, and may be coupled with $X_t$ so that $l_t$ is at least $X_t$ as long as the partition $P$ has never had a part of size at least $\frac{n}{3}$. Together, these imply that after time $cn$ there is at least a probability of $1-\eps$ that $P$ has had a part of size at least $\frac{n}{3}$, which completes the proof.
\end{proof}

Continuing, we need to find the time taken for the rest of the cards to `get random'. This happens when the partition $P(t)$ has only a single part.

\begin{Proposition}\label{prop:secondhalf}
Consider the time $s$ between the partition $P(t)$ first containing a part of size at least $\frac{n}{3}$ and the partition $P(t+s)$ becoming the singleton partition $(n)$. The expected value of this time is $n\ln(n) + O(n)$, and the variance is $O(n^2)$.
\end{Proposition}
\begin{proof}
Consider the size of the largest part of $P$. At each step, Marking Scheme \ref{mar:mult} cannot decrease the size of this largest part, and may increase it. If the largest part has size $k$, then there is a probability of at least $\frac{k(n-k)}{n(n-1)}$ that the next step increases the size of the largest part by at least one, by transposing a card not in this part with a card that is in this part. The expected time between having a part of size at least $\frac{n}{3}$ and having all of the cards in a single part is \begin{align*}&\sum_{k=\frac{n}{3}}^{n-1}\frac{n(n-1)}{k(n-k)} \\ &= (n-1)\sum_{k=\frac{n}{3}}^{n-1}\left(\frac{1}{k}+\frac{1}{n-k}\right) \\ &= (n-1)\left(\sum_{k=1}^{n-1}\frac{1}{k}+\sum_{k=\frac{n}{3}}^{\frac{2n}{3}}\frac{1}{k}\right) \\ &= (n-1)\left(\ln(n-1) + \ln(\frac{2n}{3}) - \ln(\frac{n}{3})\right) + O(n) \\ &= n\ln(n) + O(n) \end{align*}

The calculation for the variance is similar. The variance of the entire time is at most the sum from $k=\frac{n}{3}$ to $k=n-1$ of the variance of the geometric random variable with probability $\frac{k(n-k)}{n(n-1)}$. Expanding this sum in the same way gives the result.
\end{proof}

\begin{Corollary}
The separation distance mixing time of Random Walk \ref{wal:rt2} is at most $n\ln(n) + O(n)$, in the sense that for any $\eps$ there is a constant $c$ so that after time $n\ln(n) + cn$, the separation distance from the uniform distribution is at most $\eps$.
\end{Corollary}
\begin{proof}
This is a consequence of Propositions \ref{prop:firsthalf} and \ref{prop:secondhalf}, using Chebyshev's inequality with the latter.
\end{proof}

Together with the usual coupon collector lower bound of $n\ln(n) + O(n)$ (for our lazier version of the walk), this gives cutoff in separation distance for the random transposition walk. All that remains is to check the results used in Proposition \ref{prop:mixed}.

\section{The key result}
\label{sec:calcs}

In the previous section, we often found ourselves in the position of having two sets of shuffled cards of sizes $m$ and $n$ with $m \geq n$, while being about to transpose a card from one set with a card from the other set. We would like to understand what happens after such a transposition is made --- what can we say about the possible orders of all $m+n$ cards afterwards? 

The strongest result that could be hoped for would be that all $m+n$ cards are random. Not all elements of $S_{m+n}$ may be obtained by multiplying elements of $S_n$ and $S_m$ and a transposition, if the identifications of the smaller symmetric groups with subgroups of the larger are fixed, so perhaps we might hope only hope that this produces the correct distribution on cycle types of the resulting permutation. For applications to the random transposition walk, a result on cycle types will be sufficient, because this random walk is generated by a conjugacy class, so all permutations with a given cycle type are equally likely.

This result on partitions is also not true, as may be seen by considering the probability that the resulting permutation is a single $(m+n)$--cycle --- to produce an $(m+n)$--cycle, the permutations of the initial $m$ and $n$ cards should be an $m$--cycle and an $n$--cycle, which happen with probabilities $\frac{1}{m}$ and $\frac{1}{n}$, and then any transposition between the two sets will result in an $(m+n)$--cycle. But the probability that a random permutation in $S_{m+n}$ is an $(m+n)$--cycle is $\frac{1}{m+n}$.

We will prove a yet weaker result of this type. Rather than concluding that all $m+n$ cards are random, we provide a (random) algorithm for dividing the $m+n$ cards into $m+k$ cards and $n-k$ cards, for a random $k$, so that the permutations in each set have random cycle types. Because $m$ was at least $n$, this may be seen as an improvement in how much of the deck is random, and we will use this result repeatedly to bound the time taken until the entire deck is random, in Marking Scheme \ref{mar:mult} and Corollary \ref{cor:sst}. This idea is shown in Examples \ref{ex:add1}, \ref{ex:add2}, and \ref{ex:add3}.

\begin{Example}\label{ex:add1}
Consider a deck of $5$ cards, with the top four cards being the cards $1$ to $4$ in a random order, and the last card being card $5$. With equal probabilities of $\frac{1}{5}$, swap card $5$ with any other card, or leave it in place. Then the whole deck is in a random order.

Building larger and larger random permutations by using this fact repeatedly is the technique used in \cite{Broder} to give a strong stationary time for the random transposition walk.
\end{Example}

\begin{Example}\label{ex:add2}
Consider a deck of $6$ cards, with the top four cards being the cards $1$ to $4$ in a random order, and the last two cards being card $5$ and $6$ in a random order. With equal probabilities of $\frac{1}{10}$, swap card $5$ or $6$ with any of the cards $1$ to $4$, or leave the deck in its original order with probability $\frac{2}{10}$. 

Mark the cards $1$ to $4$, and mark some of the other two cards as follows. 
\begin{itemize}
\item If the cards $5$ and $6$ were in their respective positions $5$ and $6$ to start, then 
\begin{itemize}
\item With probability $\frac{1}{5}$ set $k=2$ and mark the $5$ and $6$. 
\item Otherwise (probability $\frac{4}{5}$) set $k=1$ and mark whichever of the $5$ and $6$ was moved, or one at random if the order wasn't changed. 
\end{itemize}
\item Otherwise (the cards $5$ and $6$ were in positions $6$ and $5$ to start), 
\begin{itemize}
\item Mark the $5$ and $6$.
\end{itemize}
\end{itemize}

Then conditioned on $k$, the cycle type of the permutation of the $4+k$ marked cards is in distribution the same as the cycle type of a uniformly random permutation from $S_{4+k}$. That is, this algorithm produces either $5$ or $6$ marked cards, but conditioned on this number, the distribution of their cycle types is correct. 
\end{Example}

We will use the idea of Example \ref{ex:add2} to mark cards more rapidly than Broder's scheme. We will need to be careful, because what may deduced from a set of cards being marked differs between different marking schemes, as does which events are being conditioned on in the analysis.

To show what sorts of phenomena occur when we move to larger numbers, we repeat Example \ref{ex:add2} in a situation where up to three additional cards may be marked.

\begin{Example}\label{ex:add3}
Consider a deck of $7$ cards, with the top four cards being the cards $1$ to $4$ in a random order, and the last three cards being cards $5$, $6$, and $7$, in a random order. With equal probabilities of $\frac{1}{15}$, swap card $5$, $6$, or $7$ with any of the cards $1$ to $4$, or choose one of the cards $5$, $6$, or $7$ but leave the deck in its original order. 

Mark the cards $1$ to $4$, and mark some of the other three cards as follows. 
\begin{itemize}
\item If the original permutation of the cards $5$ to $7$ was a $3$--cycle, mark all three cards and set $k=3$.
\item If the original permutation of the cards $5$ to $7$ was a $2$--cycle $\tp{i}{j}(h)$ and one of the cards $i$ and $j$ was moved or chosen but not moved, then
\begin{itemize}
\item With probability $\frac{1}{6}$ set $k=3$ and mark all three cards. 
\item Otherwise (probability $\frac{5}{6}$) set $k=2$ and mark cards $i$ and $j$.
\end{itemize}
\item If the original permutation of the cards $5$ to $7$ was a $2$--cycle $\tp{i}{j}(h)$ and $h$ was either moved or chosen but not moved, then
\begin{itemize}
\item With probability $\frac{2}{5}$ set $k=3$ and mark all three cards. 
\item Otherwise (probability $\frac{3}{5}$) set $k=1$ and mark card $h$.
\end{itemize}
\item Otherwise, if the original permutation of the cards $5$ to $7$ was the identity $(5)(6)(7)$ and card $i$ was either moved or chosen but not moved, then
\begin{itemize}
\item With probability $\frac{1}{15}$ set $k=3$ and mark all three cards.
\item With probability $\frac{1}{3}$ set $k=2$ and mark card $i$ and a random one of the other two cards.
\item Otherwise (probability $\frac{3}{5}$) set $k=1$ and mark card $i$.
\end{itemize}
\end{itemize}

Then conditioned on $k$, the cycle type of the permutation of the $4+k$ marked cards is in distribution the same as the cycle type of a uniformly random permutation from $S_{4+k}$. This algorithm produces $5$, $6$, or $7$ marked cards and conditioned on this number, the distribution of their cycle types is as that of a uniform permutation.
\end{Example}

We now move to the general case.

Somewhat awkwardly, the ideal setting for the following calculations seems to be somewhere in between considering permutations and considering their cycle types. We will work with partitions, but often various terms will be multiplied by factors indicating that the term is really counting something to do with objects with a little more structure, like a partition with a choice of part, or a partition with an order on some of its parts.

\begin{Proposition}\label{prop:combineparts}
Let $\la$ and $\mu$ be partitions of $m$ and $n$, respectively, with $m \geq n$. Create new partitions $\nu$ and $\xi$ by merging $\mu$ with $\la$, in the sense of Definition \ref{def:merge}. Note that $\nu$ is a partition of a random integer between $m+\mu_0$ and $m+n$, and $\xi$ is a partition of $m+n - |\nu|$.

If $\lambda$ and $\mu$ are the cycle types of independent uniformly random elements of $S_m$ and $S_n$, then for any fixed $k \leq n$, conditioned on $|\nu| = n+k$, the distributions of $\nu$ and $\xi$ are the distributions of cycle types of uniformly random elements of $S_{m+k}$ and $S_{n-k}$, and $\nu$ and $\xi$ are independent of one another.
\end{Proposition}

The two cases could be seen as merging a random part of $\la$ with a random part of $\mu$, choosing parts from each partition with probabilities proportional to part size, with a single chance of choosing an empty part from $\la$, with case 2 of Definition \ref{def:merge} corresponding to choosing this empty part. 

Before we embark on the proof of Proposition \ref{prop:combineparts}, we give some calculations of these probabilities in small cases. Table \ref{tab:combineparts5} shows the calculation for every possibility in the case $m = 3$, $n = 2$, and $k = 2$. An example of this size does not illustrate all possible behaviours, so Table \ref{tab:combineparts9} shows the same calculation for just a few possibilities in the case $m = 5$, $n = 4$, and $k = 4$. Both of these examples have $k = n$ and so $|\xi| = 0$. The following remark justifies choosing only examples with $|\xi| = 0$ with reference to the proof of the present proposition.

\begin{Remark}
Surprisingly, allowing $\xi$ to have size larger than zero has almost no impact on the calculations --- in the calculation we are about to start, it gives the initial factor of $$\frac{m-n+2k}{m+k}\prod_{i=1}^n \binom{b_i+c_i}{b_i},$$ which mostly cancels out when we divide by the probability of the partition $\xi$, leaving a remainder which only depends on $\nu$ and $\xi$ in that it depends on $k$, and thus is ignored when we reduce to Proposition \ref{prop:subparts}.
\end{Remark}

\bgroup
\def\arraystretch{1.3}
\begin{table}\resizebox{\textwidth}{!}{%
$\begin{array}{cccccccccc}
\nu & \xi & \la & \mu & \Pr(\la) & \Pr(\mu) & \Pr(\text{Join}) & \Pr(\text{Others}) & \text{Prob.} & \text{Total} \\
(5) & \emptyset & (3) & (2) & \f{1}{3} & \f{1}{2} & \f{3}{4} & & \f{1}{8} & \f{24}{196} \\
\hline
(4,1) & \emptyset & (3) & (1,1) & \f{1}{3} & \f{1}{2} & \f{3}{4} & \f{1}{4} & \f{1}{32} & \\
 & & (2,1) & (2) & \f{1}{2} & \f{1}{2} & \f{1}{2} & & \f{1}{8} & \f{30}{196} \\
\hline
(3,2) & \emptyset & (3) & (2) & \f{1}{3} & \f{1}{2} & \f{1}{4} & & \f{1}{24} & \\
&  & (2,1) & (2) & \f{1}{2} & \f{1}{2} & \f{1}{4} & & \f{1}{16} & \f{20}{196} \\
\hline
(3,1,1) & \emptyset & (3) & (1,1) & \f{1}{3} & \f{1}{2} & \f{1}{4} & \f{1}{4} & \f{1}{96} & \\
 & & (2,1) & (1,1) & \f{1}{2} & \f{1}{2} & \f{1}{2} & \f{1}{4} & \f{1}{32} & \\
 & & (1,1,1) & (2) & \f{1}{6} & \f{1}{2} & \f{3}{4} & & \f{1}{16} & \f{20}{196} \\
\hline
(2,2,1) & \emptyset & (2,1) & (2) & \f{1}{2} & \f{1}{2} & \f{1}{4} & & \f{1}{16} & \\
 & & (2,1) & (1,1) & \f{1}{2} & \f{1}{2} & \f{1}{4} & \f{1}{4} & \f{1}{64} & \f{15}{196} \\
\hline
(2,1,1,1) & \emptyset & (2,1) & (1,1) & \f{1}{2} & \f{1}{2} & \f{1}{4} & \f{1}{4} & \f{1}{64} & \\
 & & (1,1,1) & (2) & \f{1}{6} & \f{1}{2} & \f{3}{4} & \f{1}{4} & \f{1}{64} & \\
 & & (1,1,1) & (1,1) & \f{1}{6} & \f{1}{2} & \f{1}{4} & & \f{1}{48} & \f{10}{196} \\
\hline
(1,1,1,1,1) & \emptyset & (1,1,1) & (1,1) & \f{1}{6} & \f{1}{2} & \f{1}{4} & \f{1}{4} & \f{1}{192} & \f{1}{196} \\
\end{array}$}
\caption{The probabilities of producing each partition $\nu$ of size $5$ (and $\xi$ an empty partition) from random partitions $\la$ and $\mu$ of sizes $3$ and $2$ according to the procedure in Proposition \ref{prop:combineparts}. Notice that the probabilities for each partition, in the final column, are proportional to the size of that conjugacy class in the symmetric group $S_5$, confirming that proposition for $m = \left|\la\right| = 3$, $n = \left|\mu\right| = 2$, and $k = \left|\nu\right| - m = 2$.}
\label{tab:combineparts5}
\end{table}

\begin{table}\resizebox{\textwidth}{!}{%
$\begin{array}{ccccccccccc}
\nu & \xi & \la & \mu & \mu_0 & \Pr(\la) & \Pr(\mu) & \Pr(\text{Join}) & \Pr(\text{Others}) & \text{Prob.} & \text{Total} \\
(7,1,1) & \emptyset & (5) & (2,1,1) & & \f{1}{5} & \f{1}{4} & \f{5}{12} & \f{2}{7}\cdot\f{1}{8} & \f{1}{1344} &  \\
 &  & (4,1) & (3,1) & & \f{1}{4} & \f{1}{3} & \f{1}{2} & \f{1}{8} & \f{1}{192} &  \\
 &  & (3,1,1) & (4) & & \f{1}{6} & \f{1}{4} & \f{1}{2} & & \f{1}{48} & \f{30}{1120}  \\
\hline
(6,2,1) & \emptyset & (5) & (2,1,1) & & \f{1}{5} & \f{1}{4} & \f{5}{12} & \f{2}{6}\cdot\f{1}{8}+\f{1}{6}\cdot\f{2}{7} & \f{5}{2688} &  \\
& & (4,1) & (2,2) & & \f{1}{4} & \f{1}{8} & \f{2}{3} & \f{2}{7} & \f{1}{168} &  \\
& & (3,2) & (3,1) & & \f{1}{6} & \f{1}{3} & \f{3}{8} & \f{1}{8} & \f{1}{384} &  \\
& & (2,2,1) & (4) & & \f{1}{8} & \f{1}{4} & \f{2}{3} & & \f{1}{48} & \f{35}{1120} \\
\hline
(5,2,1,1) & \emptyset & (5) & (2,1,1) & 2 & \f{1}{5} & \f{1}{4} & \f{1}{12} & \f{2}{6}\cdot\f{1}{8}+\f{1}{6}\cdot\f{2}{7} & \f{1}{2688} &  \\
& & (5) & (2,1,1) & 1 & \f{1}{5} & \f{1}{4} & \f{1}{12} & \f{2}{7}\cdot\f{1}{8} & \f{1}{6720} &  \\
& & (4,1) & (2,1,1) & & \f{1}{4} & \f{1}{4} & \f{1}{3} & \f{2}{6}\cdot\f{1}{8}+\f{1}{6}\cdot\f{2}{7} & \f{5}{2688} &  \\
& & (3,2) & (2,1,1) & & \f{1}{6} & \f{1}{4} & \f{1}{4} & \f{2}{7}\cdot\f{1}{8} & \f{1}{2688} &  \\
& & (3,1,1) & (2,2) & & \f{1}{6} & \f{1}{8} & \f{1}{2} & \f{2}{7} & \f{1}{336} &  \\
& & (2,2,1) & (3,1) & & \f{1}{8} & \f{1}{3} & \f{1}{2} & \f{1}{8} & \f{1}{384} &  \\
& & (2,1,1,1) & (4) & & \f{1}{12} & \f{1}{4} & \f{1}{2} & & \f{1}{96} & \f{21}{1120} \\
\end{array}$}
\caption{The probabilities of producing some partitions $\nu$ of size $9$ (and $\xi$ an empty partition) from random partitions $\la$ and $\mu$ of sizes $5$ and $4$ according to the procedure in Proposition \ref{prop:combineparts}. The probabilities for each partition, in the final column, are proportional to the size of the corresponding conjugacy class in the symmetric group $S_9$. The cases in this table are all the possible partitions $\nu$ which can be made from $\la = (5)$ and $\mu = (2,1,1)$. These are chosen to show the role of the third-to-last column, labelled $\Pr(\text{Others})$, because the example in Table \ref{tab:combineparts5} was not large enough to show the range of behaviour that may be exhibited by this term. In each case, this is the probability of moving each of the other parts of $\mu$ to $\nu$. In the notation of the calculations in the proof of Proposition \ref{prop:combineparts}, this is the `including factor' $I(b)$.}
\label{tab:combineparts9}
\end{table}
\egroup

\begin{proof}[Proof of Proposition \ref{prop:combineparts}]
We directly compute the probability that any given pair of partitions $\nu$ and $\xi$ are produced by this process, for fixed $m$ and $n$. We consider the possible choices of $\la$ and $\mu$ which could produce these $\nu$ and $\xi$, as well as a choice of $\mu_0$. In case $1$ of Definition \ref{def:merge}, the value of $\mu_0$ is determined by the choice of $\la$ and $\mu$, while in case $2$, $\mu_0$ may be any part of $\mu$. The partition $\la$ is a union of parts of $\nu$, potentially with one of them reduced in size (corresponding to case 1). The partition $\mu$ is comprised of all remaining parts of $\nu$, all parts of $\xi$, and in case $1$, another part whose size is the amount by which the part in $\la$ was reduced.

We will need the following notation. Let $|\nu| = m+k$ and $|\xi| = n-k$. In case $1$, let $\la_0$ be the size of the part of $\la$ with which $\mu_0$ was combined. In case $2$, there is no such part. 

For concreteness, for each $i$ let $a_i$ be the number of parts of size $i$ in $\la$, $b_i$ be the number of parts of size $i$ in $\mu$ which were added to $\nu$, and $c_i$ the number of parts of size $i$ in $\mu$ which were added to $\xi$, not counting the part $\mu_0$ of $\mu$ and in case $1$, not counting the part $\la_0$ of $\la$. 

Let $I(b)$ be the probability of moving a specific collection of $b_i$ parts of each size $i$ from $\mu$ to $\nu$. The probability of moving no more parts after these is $\frac{m-n+2k}{m+k}$. If there were more parts of these sizes in $\mu$, some of which were moved to $\nu$ and some of which were not, then the probability that the parts moved from $\mu$ to $\nu$ are exactly $b_i$ parts of each size $i$ is $$I(b)\frac{m-n+2k}{m+k}\prod_{i=1}^n \binom{b_i+c_i}{b_i}.$$ In the language of Remark \ref{rem:include}, $I(b) = f(m,\mu',\mu_0)$, where $\mu'$ is the subpartition of $\mu$ comprised of $\mu_0$ and $b_i$ parts of each size $i$.

In the following, we will use $\Pr$ for the probability that a certain event occurs in the process we are analysing, and $\Pre$ for the \textbf{ex}act probability of an event, derived from a uniform distribution on permutations. The distribution of cycle types in uniformly random permutations is given in Lemma \ref{lem:cycletype}. Products $\prod_p$ are over all possible sizes $p$ of partition parts.

Rather than repeat similar calculations for terms corresponding to cases $1$ and $2$, we give the more general expression, for case $1$, and describe how it must be modified in case $2$. In case $2$, there is no part $\la_0$, so terms depending on $\la_0$ should be omitted from the following expressions. These terms cancel out quickly, and the resulting expressions are correct in both cases. For the individual terms to be correct in the initial expressions, in case $2$ the term $\la_0(a_{\la_0} + 1)$ should be interpreted as $1$, and the index variable $p$ should never be equal to $\la_0$, so products $\prod_{p \neq \la_0}$ do not miss any terms, and terms $\left[\cdots\right]_{p=\la_0}$ are ignored. 

To produce the specific partitions $\nu$ and $\xi$, several things need to happen. We need to start with appropriate partitions $\la$ and $\mu$, choose the correct parts of $\la$ and $\mu$ to combine, if any, including a choice of $\mu_0$ in case $2$, and then decide how the other parts of $\mu$ should be distributed between $\nu$ and $\xi$. Finally, the probabilities of these sequences of events must be summed over the various choices of $\la$, $\mu$, and $\mu_0$ which could produce the required $\nu$ and $\xi$. We compute as follows:

\begin{dgroup*}
\begin{dmath*}\Pr(\nu,\xi) = \sum_{\la,\mu,\mu_0}\Pre(\text{Choose $\la,\mu$})\Pr(\text{Combine correct parts})\Pr(\text{Parts to }\mu)\Pr(\text{Parts to }\xi)
\end{dmath*}
\begin{dmath*}
= \sum_{\la,\mu,\mu_0}\left(\Pre(\la)\Pre(\mu)\right)\left(\frac{\la_0(a_{\la_0}+1)\mu_0(b_{\mu_0}+c_{\mu_0}+1)}{n(m+1)}\right)\cdot\left(I(b)\prod_{p} \binom{b_p+c_p}{b_p}\right)\left(\frac{m-n+2k}{m+k}\right)
\end{dmath*}
\begin{dmath*}
= \sum_{\la,\mu,\mu_0}\left(\frac{1}{\prod_{p \neq \la_0}p^{a_p}a_p!\left[p^{a_p+1}(a_p+1)!\right]_{p=\la_0}}\cdot\frac{1}{\prod_{p \neq \mu_0}p^{b_p+c_p}(b_p+c_p)!\left[p^{b_p+c_p+1}(b_p+c_p+1)!\right]_{p=\mu_0}}\cdot\left(\frac{\la_0(a_{\la_0}+1)\mu_0(b_{\mu_0}+c_{\mu_0}+1)}{n(m+1)}\right)\cdot I(b)\cdot\prod_{p} \binom{b_p+c_p}{b_p}\cdot \left(\frac{m-n+2k}{m+k}\right)\right)
\end{dmath*}
\begin{dmath*}
= \sum_{\la,\mu,\mu_0}\left(\frac{1}{\prod_{p}p^{a_p}a_p!}\cdot\frac{1}{\prod_{p}p^{b_p+c_p}(b_p+c_p)!}\cdot\frac{1}{\la_0 (a_{\la_0}+1)\mu_0 (b_{\mu_0}+c_{\mu_0}+1)}\cdot\left(\frac{\la_0(a_{\la_0}+1)\mu_0(b_{\mu_0}+c_{\mu_0}+1)}{n(m+1)}\right)\cdot I(b)\cdot\prod_{p} \binom{b_p+c_p}{b_p}\cdot \left(\frac{m-n+2k}{m+k}\right)\right)
\end{dmath*}
\begin{dmath*}
= \sum_{\la,\mu,\mu_0}\left(\frac{1}{\prod_{p}p^{a_p}a_p!}\cdot\frac{\prod_{p}\binom{b_p+c_p}{b_p}}{\prod_{p}p^{b_p+c_p}(b_p+c_p)!}\cdot I(b) \cdot \left(\frac{m-n+2k}{(m+k)(n(m+1))}\right)\right)
\end{dmath*}
\end{dgroup*}

We now divide this probability by the probabilities of getting $\nu$ and $\xi$ as the cycle types of uniform elements of $S_{m+k}$ and $S_{n-k}$. Showing that this quotient does not depend on $\nu$ or $\xi$ will complete the proof (The quotient would be equal to $1$ if we conditioned on $k$). In the following, terms corresponding to case 2 are described by taking $\la_0 = 0$, so that in all cases, $\la_0+\mu_0$ is the size of the part containing $\mu_0$ after this part has been added to $\nu$, whether it was combined with an existing part or not.

We divide by the probability 

\begin{dgroup*}
\begin{dmath*}
\Pre(\nu)\Pre(\xi) = \frac{1}{\prod_{p \neq \la_0+\mu_0}p^{a_p+b_p}(a_p+b_p)!\cdot \left[p^{a_p+b_p+1}(a_p+b_p+1)!\right]_{p=\la_0+\mu_0}} \cdot \frac{1}{\prod_{p}p^{c_p}(c_p)!}
\end{dmath*}
\begin{dmath*}
 = \frac{1}{\prod_{p}p^{a_p+b_p}(a_p+b_p)!} \cdot \frac{1}{(\la_0+\mu_0)(a_{\la_0+\mu_0} + b_{\la_0+\mu_0} + 1)}\cdot \frac{1}{\prod_{p}p^{c_p}(c_p)!}
\end{dmath*}
\end{dgroup*}

This gives the ratio

\begin{dgroup*}
\begin{dmath*}
\frac{\Pr(\nu,\xi)}{\Pre(\nu)\Pre(\xi)} = \sum_{\la,\mu,\mu_0}\left(\frac{\prod_{p}p^{a_p+b_p}(a_p+b_p)!}{\prod_{p}p^{a_p}a_p!}\cdot\frac{\prod_{p}\left(\binom{b_i+c_i}{b_i}p^{c_p}(c_p)!\right)}{\prod_{p}p^{b_p+c_p}(b_p+c_p)!}\cdot(\la_0+\mu_0)(a_{\la_0+\mu_0} + b_{\la_0+\mu_0} + 1)\cdot I(b)\cdot\frac{m-n+2k}{(m+k)(n(m+1))}\right)
\end{dmath*}
\begin{dmath*}
 = \sum_{\la,\mu,\mu_0}\left(\prod_{p}\left(\binom{a_p+b_p}{a_p}\right)(\la_0+\mu_0)(a_{\la_0+\mu_0} + b_{\la_0+\mu_0} + 1)\cdot I(b)\cdot\frac{m-n+2k}{(m+k)(n(m+1))}\right)
\end{dmath*}
\end{dgroup*}

Noting that our goal is just to show that this probability does not depend on the partitions $\nu$ and $\xi$, we are left to check Proposition \ref{prop:subparts}, and then the proof is complete.

\begin{Proposition}\label{prop:subparts}
For partitions $\nu$ and $\xi$ as in Proposition \ref{prop:combineparts}, the sum over all partitions $\la$ and $\mu$ of fixed sizes $m$ and $n$, and in case 2, also over a choice of $\mu_0$, of \begin{equation}\label{eq:weight}\prod_{p}\left(\binom{a_p+b_p}{a_p}\right)(\la_0+\mu_0)(a_{\la_0+\mu_0} + b_{\la_0+\mu_0} + 1)\cdot I(b)\end{equation} is equal to $|\nu|$. In particular, its dependence on $\nu$ and $\xi$ is only on the size of $\nu$.

Here, notation is as used in that result, so $\la$ and $\mu$ are partitions so that 
\begin{itemize}
\item $\nu$ is obtained by adding some parts of $\mu$ to $\la$, possibly merging one with an existing part $\la_0$
\item $\xi$ is comprised of the remaining parts of $\mu$
\item $(\la_0+\mu_0)$ is the size of the combined part in $\nu$, or if there is none such, of an arbitrary part $\mu_0$ of $\nu$ which was moved from $\mu$
\item $(a_{\la_0+\mu_0} + b_{\la_0+\mu_0} + 1)$ is the number of parts of size $\la_0 + \mu_0$ in $\nu$
\item $I(b)$ is the including factor of the parts of $\mu$ included in $\nu$ (Definition \ref{def:include}). (It doesn't matter if we demand that the remaining parts are excluded, as this results in a factor which is a constant --- it does not depend on $\nu$ and $\xi$ beyond dependence on their sizes via $k$.)
\item $a_p$ and $b_p$ are the number of parts of size $p$ in $\nu$ which came from $\la$ and from $\mu$ respectively, not counting the combined part.
\end{itemize}
\end{Proposition}

To show how much simpler the expressions of Proposition \ref{prop:subparts} are than those of Proposition \ref{prop:combineparts}, Tables \ref{tab:combineparts52} and \ref{tab:combineparts92} show the calculations required to verify Proposition \ref{prop:subparts} in the same cases as Tables \ref{tab:combineparts5} and \ref{tab:combineparts9} for Proposition \ref{prop:combineparts}.

\bgroup
\def\arraystretch{1.3}
\begin{table}\resizebox{\textwidth}{!}{%
$\begin{array}{cccccccccccc}
\nu & \xi & \la & \mu & \la_0 & \mu_0 & \binom{a_1+b_1}{a_1} & \la_0 + \mu_0 & \#(\la_0+\mu_0) & I(b) & \text{Product} & \text{Total} \\
(5) & \emptyset & (3) & (2) & 3 & 2 & & 5 & & & 5 & 5\\
\hline
(4,1) & \emptyset & (3) & (1,1) & 3 & 1 & & 4 & & \f{1}{4} & 1 & \\
 & & (2,1) & (2) & 2 & 2 & & 4 & & & 4 & 5\\
\hline
(3,2) & \emptyset & (3) & (2) & & 2 & & 2 & & & 2 & \\
&  & (2,1) & (2) & 1 & 2 & & 3 & & & 3 & 5\\
\hline
(3,1,1) & \emptyset & (3) & (1,1) & & 1 & & 1 & 2 & \f{1}{4} & \f{1}{2} & \\
 & & (2,1) & (1,1) & 2 & 1 & 2 & 3 & & \f{1}{4} & \f{3}{2} & \\
 & & (1,1,1) & (2) & 1 & 2 & & 3 & & & 3 & 5\\
\hline
(2,2,1) & \emptyset & (2,1) & (2) & & 2 & & 2 & 2 & & 4 & \\
 & & (2,1) & (1,1) & 1 & 1 & & 2 & 2 & \f{1}{4} & 1 & 5\\
\hline
(2,1,1,1) & \emptyset & (2,1) & (1,1) & & 1 & 2 & 1 & 3 & \f{1}{4} & \f{3}{2} & \\
 & & (1,1,1) & (2) & & 2 & & 2 & & & 2 & \\
 & & (1,1,1) & (1,1) & 1 & 1 & 3 & 2 & & \f{1}{4} & \f{3}{2} & 5\\
\hline
(1,1,1,1,1) & \emptyset & (1,1,1) & (1,1) & & 1 & 4 & 1 & 5 & \f{1}{4} & 5 & 5\\
\end{array}$}
\caption{Verification of Proposition \ref{prop:subparts} for partitions $\nu$ and $\xi$ of sizes $5$ and $0$, built from partitions $\la$ and $\mu$ of sizes $3$ and $2$. Compare to Table \ref{tab:combineparts5}. Empty cells indicate either that $\la_0$ is undefined or that a factor of $1$ has been omitted for clarity. The heading $\#(\la_0+\mu_0)$ is short for $a_{\la_0+\mu_0} + b_{\la_0+\mu_0} + 1$. The result is that all entries in the final column are equal.}
\label{tab:combineparts52}
\end{table}

\begin{table}\resizebox{\textwidth}{!}{%
$\begin{array}{cccccccccccc}
\nu & \xi & \la & \mu & \la_0 & \mu_0 & \binom{a_1+b_1}{a_1} & \la_0 + \mu_0 & \#(\la_0+\mu_0) & I(b) & \text{Product} & \text{Total} \\
(7,1,1) & \emptyset & (5) & (2,1,1) & 5 & 2 & & 7 & & \f{2}{7}\cdot\f{1}{8} & \f{1}{4} & \\
 &  & (4,1) & (3,1) & 4 & 3 & 2 & 7 & & \f{1}{8} & \f{1}{4} & \\
 &  & (3,1,1) & (4) & 3 & 4 & & 7 & & & \f{1}{4} & 9\\
\hline
(6,2,1) & \emptyset & (5) & (2,1,1) & 5 & 1 & & 6 & & \f{2}{6}\cdot\f{1}{8}+\f{1}{6}\cdot\f{2}{7} & \f{1}{4} & \\
& & (4,1) & (2,2) & 4 & 2 & & 6 & & \f{2}{7} & \f{1}{4} & \\
& & (3,2) & (3,1) & 3 & 3 & & 6 & & \f{1}{8} & \f{1}{4} & \\
& & (2,2,1) & (4) & 2 & 4 & & 6 & & & \f{1}{4} & 9\\
\hline
(5,2,1,1) & \emptyset & (5) & (2,1,1) & & 2 & & 2 & & \f{2}{7}\cdot\f{1}{8} & \f{1}{14} & \\
& & (5) & (2,1,1) & & 1 & & 1 & 2 & \f{2}{6}\cdot\f{1}{8}+\f{1}{6}\cdot\f{2}{7} & \f{1}{12}+\f{2}{21} & \\
& & (4,1) & (2,1,1) & 4 & 1 & 2 & 5 & & \f{2}{6}\cdot\f{1}{8}+\f{1}{6}\cdot\f{2}{7} & \f{10}{24}+\f{10}{21} & \\
& & (3,2) & (2,1,1) & 3 & 2 & & 5 & & \f{2}{7}\cdot\f{1}{8} & \f{5}{28} & \\
& & (3,1,1) & (2,2) & 3 & 2 & & 5 & & \f{2}{7} & \f{10}{7} & \\
& & (2,2,1) & (3,1) & 2 & 3 & 2 & 5 & & \f{1}{8} & \f{5}{4} & \\
& & (2,1,1,1) & (4) & 1 & 4 & & 5 & & & 5 & 9\\
\end{array}$}
\caption{Verification of Proposition \ref{prop:subparts} for some partitions $\nu$ and $\xi$ of sizes $9$ and $0$, built from partitions $\la$ and $\mu$ of sizes $5$ and $4$. Compare to Table \ref{tab:combineparts9}. Empty cells indicate either that $\la_0$ is undefined or that a factor of $1$ has been omitted for clarity. The heading $\#(\la_0+\mu_0)$ is short for $a_{\la_0+\mu_0} + b_{\la_0+\mu_0} + 1$. Again, all entries in the final column are equal.}
\label{tab:combineparts92}
\end{table}
\egroup

To prove Proposition \ref{prop:subparts}, we first note that the claim in this result does not depend on $\xi$. Adding or removing a part of any size to or from both $\xi$ and each choice of $\mu$ does not affect any of the terms in the expression, so we may assume that $\xi$ is empty and that this is a question just about breaking up a partition $\nu$ of size $m+k$ into parts of size $m = |\la|$ and $k = |\mu|$.

For each choice of $\la$, $\mu$, and $\mu_0$, let the weight be the sum of the corresponding terms in Equation \ref{eq:weight}. We need to show that the sum of the weights of all choices of $\la$, $\mu$, and $\mu_0$ is $m+k = |\nu|$.   

When $k=1$, the sum of the weights is $m+k$, because a choice of $\la$, $\mu$, and $\mu_0$ is specified by the size of the combined part, and its weight is the total size of parts of that size.

Some choices of $\la$ and $\mu$ have $\mu$ having only a single part, of size $k$. The total weight of these is the sum of the sizes of parts of $\nu$ of size at least $k$. 

The other choices of $\la$ and $\mu$ have $\mu$ having more than one part. We divide terms contributing to the weight of these choices according to which part from $\nu$ was the last to be included in $\nu$. Consider all terms where a part of size $i$ is the last to be considered, with $\mu$ initially having $b_i$ parts of size $i$, not counting $\mu_0$ even if it was of that size. These terms contribute $\frac{ib_i}{m+k-i}$ times the weight of the smaller configuration where a partition of $m+k-i$ is broken into $\la'$ of size $m$ and $\mu'$ of size $k-i$. By induction, the sum of such weights is $m+k-i$, so these terms contribute a total of $ib_i$. Adding these terms over all choices of $i$, we get the sum of the sizes of parts of size less than $k$.

Combining these two cases gives that the sum of weights of all choices of $\la$, $\mu$, and $\mu_0$ is the sum of the sizes of all parts of $\nu$, which completes the proof of Proposition \ref{prop:subparts}.  

Thus we have verified Proposition \ref{prop:combineparts}.
\end{proof}

\section{Further work}

A natural generalisation of the random transposition walk is to, at each step, choose $k$ cards and randomise them among their current positions. When $k=2$, this is the lazy random transposition walk. It is possible to construct an analogue of Broder's strong stationary time for this walk (Sections 5.4 and 5.5 of \cite{GWThesis}, but it seems likely that to prove cutoff would require an improvement by another factor of two, as is done for the random transposition walk in the present paper. The difficulty lies in finding the appropriate generalisations of Definition \ref{def:merge} and Proposition \ref{prop:mixed}.

\bibliographystyle{plain}
\bibliography{bib}

\begin{thebibliography}{1}

\bibitem{Broder}
Andrei Broder.
\newblock unpublished thesis, 1985.

\bibitem{DS}
Persi Diaconis and Mehrdad Shahshahani.
\newblock Generating a random permutation with random transpositions.
\newblock {\em Probability Theory and Related Fields}, 57(2):159--179, 1981.

\bibitem{Matthews}
Peter Matthews.
\newblock A strong uniform time for random transpositions.
\newblock {\em Journal of Theoretical Probability}, 1(4):411--423, 1988.

\bibitem{GWThesis}
Graham White.
\newblock Combinatorial methods in markov chain mixing.
\newblock Ph.D. thesis, 2017.

\end{thebibliography}
\end{document}